\documentclass[10pt]{article}
\usepackage{amssymb}
\usepackage{amsmath}
\usepackage{amsthm}
\usepackage{verbatim}
\begin{document}
\title
{Global wellposedness for a certain class of large initial data for the 3D
Navier-Stokes Equations}
\author{Percy Wong}

\theoremstyle{plain} 
\newtheorem{thm}{Theorem}
\newtheorem{prop}{Proposition}
\newtheorem{lem}{Lemma}
\newtheorem{cor}{Corollary}

\theoremstyle{definition}
\newtheorem{defn}{Definition}
\newtheorem{rmk}{Remark}
\maketitle

\begin{abstract}
In this article, we consider a special class of initial data to the 3D Navier-Stokes equations on the torus, in which there is a certain degree of orthogonality in the components of the initial data.  We showed that, under such conditions, the Navier-Stokes equations are globally wellposed.  We also showed that there exists large initial data, in the sense of the critical norm $B^{-1}_{\infty,\infty}$ that satisfies the conditions that we considered.
\end{abstract}

\section{Introduction and Summary}
In this paper, we study the global wellposedness problem on the 3-torus with
side length $1$, $\mathbb{T}^3$, for the Navier-Stokes system:

\begin{equation}\label{NS}
 \Bigg\{\begin{array}{rl}
    u_t + u\cdot\nabla u - \Delta u + \nabla p = 0 & \text{in }
\mathbb{T}^3\times\mathbb{R}_+\\
    \textrm{div }u = 0 & \text{in } \mathbb{T}^3\times\mathbb{R}_+\\
    u(x,0) = u_0(x) & \text{on } \mathbb{T}^3\times\{t=0\}
   \end{array}
\end{equation}

with the restriction on $u_0 = (u_0^1, u_0^2, u_0^3)$ that the frequency
support of the three components are disjoint, and that they are sufficiently
separated in the frequency space.

The goal of this paper is to answer in the affirmative the global wellposedness
problem above:

\begin{thm} \label{thm:main}
 Suppose $u_0^i$ are a mean zero functions with disjoint frequency support,
i.e. $\emph{supp}(\widehat{u_i}) \cap \emph{supp}(\widehat{u_j}) =
\emptyset$ for $i \neq j$. Suppose furthermore that, two of the three
components, say $u_0^1$ and $u_0^2$ are frequency localized at scale $N$ and
$N^{1+\epsilon}$ respectively for some $\epsilon > 0$, i.e.
$|\emph{supp}(\widehat{u_i})| \subset [N,O(1)N]$ and
$|\emph{supp}(\widehat{u_i})| \subset
[N^{1+\epsilon},O(1)N^{1+\epsilon}]$.  Then there exists
$\delta(\epsilon) > 0$ and some universal constant $C$ such that
if $||u_0^i||_{L^2} \leq (C^{-1} \delta \log N^{1+\epsilon})^{1/2}$ for $1\leq
i\leq 3$, and $||u_0^3||_{L^3}\leq C^{-1} N^{\frac{1}{3}-\delta(1+\epsilon)}
(\log N^{1+\epsilon})^{-1/2}$, the problem (\ref{NS}) admits a unique global
solution in $C_b(\mathbb{R}_+, L^3)$.
\end{thm}

\begin{rmk}
 As we shall see in the proof, the sole purpose of the parameter $\epsilon$
in the statement of the theorem is to allow us largeness in all three
components of the initial data.  If one is only interested in a global
existence result that allows two of the three components to be large and the
remaining one small in $B^{-1}_{\infty, \infty}$ (but still large in the lower
critical Besov norms), one can remove the $\epsilon$ in the statement and
proceed as in the proof below easily.  However, as of now, there is no
satisfactory small data wellposedness results for the 3D Navier-Stokes in
$B^{-1}_{\infty, \infty}$ and given the work of Bourgain and Pavlovi\'c
\cite{BP}, it is unlikely that one can be obtained through a classical
perturbative argument.  As a result, even if one of the components remain small
in $B^{-1}_{\infty, \infty}$, the data will not fall within the
establised framework of a small perturbation of a 2D data.
\end{rmk}

\begin{rmk}
 The reason for the problem to be posed on the torus instead of $\mathbb{R}^3$
is to avoid having initial data with infinite energy.  If one is allowed that,
the result can be extended to $\mathbb{R}^3$.
\end{rmk}

We shall also show that the conditions set out in the theorem above admits
arbitrarily large data in $B^{-1}_{\infty, \infty}$ that is not a perturbation
of a 2D data:

\begin{thm}
 For any $M > 0$, there exists initial data $u_0^1, u_0^2, u_0^3$ satisfying
that conditions in Theorem \ref{thm:main} such that each component $u_0^i$ is
large in the following sense:
\begin{displaymath}
 ||u_0^i||_{B^{-1}_{\infty,\infty}} > M
\end{displaymath}
\end{thm}

The study of wellposedness of the Navier-Stokes equations is one with a long
history, originated by the seminal paper of Leray \cite{L}, proving the
existence of weak solutions.  The study of the equations for small data in
critical spaces started with a groundbreaking paper by Fujita and Kato
\cite{FK}, in which they showed local wellposedness and small data global
wellposedness in the space $H^{\frac{1}{2}}$.  Later, Kato \cite{K} also showed
the same results in the larger critical space $L^3$.  In the early nineties,
Cannone, Meyer and Planchon \cite{CMP} showed small data global wellposedness in
the spaces $B^{-1+\frac{3}{p}}_{p,\infty}$ and more recently Koch and Tataru
\cite{KT} showed the samel result in the spaces of derivatives of $BMO$
functions $\partial BMO$.  The largest critical space known is the space
$B^{-1}_{\infty, \infty}$, for which Bourgain and Pavlovi\'c \cite{BP} have
provided an example with arbitrarily small norm that blows up in arbitrarily
short amount of time. 

To summarize we have the following embeddings of critical spaces:
\begin{displaymath}
 H^{1/2} \hookrightarrow L^3 \hookrightarrow B^{-1+\frac{3}{p}}_{p,\infty}
\hookrightarrow \partial BMO \hookrightarrow B^{-1}_{\infty, \infty}
\end{displaymath}
where small data global wellposedness have been proven for all except the last
space, in which illposedness (in the sense that arbitrarily small data can grow
to arbitrarily large magnitude in arbitrarily small time) exists.

The study for large initial data problems, where the large initial data does not
reduce to a perturbation of the 2D problems or a formulation where exact
solutions or global wellposedness is already known, started with a recent paper
by Chemin and Gallagher \cite{CG1}, and continued by Bahouri, Chemin, Gallagher, Mullaert, Paicu and Zhang (e.g. \cite{CG2}, \cite{CG3}, \cite{CGP}, \cite{CGZ}, \cite{BG}, \cite{CGM}).  

\section{Preliminary considerations}
Before we begin proving the theorems, let us first make some simple
observations that are going to greatly reduce the complexity of the problem at
hand.  First of all, the incompressibility condition in frequency space reads:
\begin{displaymath}
 \sum_{i=1}^3 \xi_i\hat{u^i}(\xi) = 0 
\end{displaymath}
Under the conditions that initially the frequency supports of the three
components of the data are disjoint, we can infer that for $1\leq i \leq3$,
\begin{displaymath}
 \xi_i\hat{u^i}(\xi) = 0
\end{displaymath}
Interpreting the above in physical space, we have, initially the data $u_0^1$
is a function of $x_2$ and $x_3$ only, $u_0^2$ a function of $x_1$ and $x_3$
only and $u_0^3$ a function of $x_1$ and $x_2$ only.  We shall therefore make
use of this observation and forget about the disjointness of the frequency
support from now on.

Let us from now on denote the Leray projector onto the divergence free vector
field by \textbf{P}.  The problem (\ref{NS}) can be simplified by breaking it
apart into four problems.  The first three are two-dimensional Navier-Stokes
equations with each component as its initial condition.  For example the first
component $u_0^1$ would yield:

\begin{equation}\label{2DNS}
 \Bigg\{\begin{array}{rl}
    v^1_t + \textbf{P}(v^1\cdot\nabla v^1) - \Delta v^1 = 0 & \text{in }
\mathbb{T}^3\times\mathbb{R}_+\\
    \textrm{div }v^1 = 0 & \text{in } \mathbb{T}^3\times\mathbb{R}_+\\
    v^1(x,0) = (u_0^1(x), 0, 0)^t & \text{on } \mathbb{T}^3\times\{t=0\}
   \end{array}
\end{equation}

Note that this is possible because $(u_0^1,0,0)$ satisfies the
incompressibility condition and similarly for the other two components.  Upon
further inspection, it is clear that the nonlinear term $\textbf{P}(u\cdot\nabla
u)$ vanishes for all time as $\frac{\partial u_0^1}{\partial x_1} = 0$ for all
$x$. Thus, the problem (\ref{2DNS}) is just a heat equation and solution is
$v^1(x,t) = (e^{t\Delta} u_0^1, 0, 0)^t$.  To make the observation above
rigorous, one only needs to invoke the uniqueness of solutions to the 2D
Navier-Stokes equations and observe that the solution to the 2D heat
equation also solves the 2D Navier-Stokes for our prescribed initial conditions.

The fourth equation, which will be the main focus of the rest of this paper,
accounts for the interaction of these solutions to heat equations:

\begin{equation}\label{R}
 \Bigg\{\begin{array}{rl}
    R_t + \textbf{P}(R\cdot\nabla R) - \Delta R + Q(v^\star, R) = F & \text{in }
\mathbb{T}^3\times\mathbb{R}_+\\
    \textrm{div }R = 0 & \text{in } \mathbb{T}^3\times\mathbb{R}_+\\
    R(x,0) = 0 & \text{on } \mathbb{T}^3\times\{t=0\}
   \end{array}
\end{equation}

where, borrowing the notation from \cite{CG1}, $Q(a,b) =
\textbf{P}\textrm{div}(a\otimes b + b \otimes a)$, $F = -\sum_{i < j} Q(v^i,
v^j)$, $v^i$ the solution to the $i$-th heat equation as above, and $v^\star =
\sum_i v^i$.

The equation above is the usual incompressible Navier-Stokes with a linear
term and an inhomogenous forcing term.  For small data, Chemin and Gallagher
\cite{CG1} has shown that the system is globally wellposed given good decay
properties of $v^\star$.  The reason is firstly the system retains all the
scaling properties of Navier-Stokes; and secondly under small data, the system
is govern by the linear term, and the decay properties of $v^\star$ will
guarantee that the evolution of the equation remains small in the suitable
norms.  This formulation thus suggests that, if we are able to control the
behaviour of $F$, that is if we can control the interaction of the solutions of
heat solutions which propagates in orthogonal directions, then a perturbative
argument will yield us global wellposedness of the original Navier-Stokes
problem.

Lastly, as we see that a good understanding of the decay property of the heat
kernel is required to solve the problem at hand, we state without proof the
following lemma for the readers' convenience.  The lemma can be proven easily by
analysing the fundamental solution of the heat equation.

\begin{lem}
 We have the following $L^p$ to $L^q$ estimates for the 3-dimensional heat
equation:
\begin{displaymath}
 ||\partial^s e^{t\Delta} u||_{L^q} \leq t^{-\frac{3+s}{2} + \frac{3}{2}
(\frac{1}{p} - \frac{1}{q} + 1)} ||u||_{L^p}
\end{displaymath}

\end{lem}

\section{Notations}
In this paper, we will be working in $L^p$ spaces, the Sobolev spaces $W^{1,p}$,
as well as various Besov spaces $B^s_{p,q}$.  Let us take a moment to establish
some simple definitions and facts about these spaces.

\begin{defn}
 Let $P_j$ denote the frequency cut-off operator of the $j$-th dyadic
frequencies, 
\begin{displaymath}
 \widehat{P_jf}(\xi) = \chi(\xi/2^j) \hat{f}(\xi)
\end{displaymath}
where $\chi(\xi)$ is a function supported on $1/2 \leq |\xi| \leq 2$ and for all
$\xi \neq 0$,
\begin{displaymath}
 \sum_{j\in\mathbb{Z}} \chi(\xi/2^j) = 1
\end{displaymath}
We say $f$ belongs to the Besov space $B^s_{p,q}(\mathbb{T}^3)$,
$s\in\mathbb{R}, 1 \leq p,q \leq \infty$, if $f$ is a distribution such that
\begin{displaymath}
 ||f||_{B^s_{p,q}} := ||2^{js}||P_jf||_{L^p}||_{l^q_j(\mathbb{Z})} < \infty
\end{displaymath}
An equivalent definition of the Besov spaces on $\mathbb{T}^3$ and for $s$
negative is
\begin{displaymath}
 ||f||_{B^s_{p,q}} := ||t^{-s/2}||e^{t\Delta}f||_{L^p}||_{L^q(\mathbb{R}_+,
\frac{dt}{t})} < \infty
\end{displaymath}
\end{defn}

The proof for the equivalence of definitions and many other useful properties of
Littlewood-Paley decompositions and Besov spaces can be found in \cite{D}.  For
the purpose of this paper, we will only require the following embeddings:

\begin{lem} \label{lem:incl}
 In $\mathbb{T}^3$, we have the inclusion of the following spaces:
\begin{displaymath}
 L^3 \hookrightarrow B^{-1+\frac{3}{p}}_{p,2} (p > 3) \hookrightarrow
B^{-1}_{\infty,\infty}
\end{displaymath}
In $\mathbb{T}^2$, we have the following inclusions:
\begin{displaymath}
 L^2 \hookrightarrow B^{-1}_{\infty,2}
\end{displaymath}
\end{lem}

The proof of the above lemma is a straightforward application of the Bernstein's
inequality and will be omitted.

\section{Estimates for $F$}

We start by estimating the $L^1_tL^3_x$ norm of $F$.  Given theorem
\ref{thm:CG} and lemma \ref{lem:incl}, this is indeed the correct space to
consider.  The terms in $F$ is of the form $e^{t\Delta}u^i \partial_i
e^{t\Delta}u^j$ where $i \neq j$.  By the triangle inequality, it is enough to
establish a bound on the $L^1_tL^3_x$ norm of $e^{t\Delta}u^i \partial_i
e^{t\Delta}u^j$, $i \neq j$.  We will consider the case where $u^i$ concentrates
around frequency $N$ and $u^j$ concentrates around frequency $N^{1+\epsilon}$
and the case with $u^i$ concentrates around $N^{1+\epsilon}$ and $u^j$
concentrates around $N$.  The remaining cases where at least one of $u^i$ or
$u^j$ is of low frequency can be estimated similarly and shall be omitted. 
Without loss of generality, we shall assume $i=1$ and $j=2$.

\textbf{Case 1} $|\text{supp}(\widehat{u^1})| \subset [N,O(1)N]$,
$|\text{supp}(\widehat{u^2})| \subset [N^{1+\epsilon},O(1)N^{1+\epsilon}]$
\begin{align} \nonumber
 \qquad &\int_0^\infty ||e^{t\Delta}u^1 \partial_1 e^{t\Delta}u^2||_{L^3} dt\\
\nonumber
& \leq \int_0^\infty ||e^{t\Delta}u^1||_{L^3_{x_2}L^3_{x_3}} ||\partial_1
e^{t\Delta}u^2||_{L^3_{x_1}L^\infty_{x_3}} dt\\ \nonumber
& \lesssim \int_0^K ||u^1||_{L^3} t^{-1/2} N^{\frac{1+\epsilon}{3}}
||u^2||_{L^3} dt + \int_K^\infty ||u^1||_{L^3} ||\partial^{4/3} e^{t\Delta}
u^2||_{L^3} dt\\ \nonumber
& \lesssim ||u^1||_{L^3} ||u^2||_{L^3} K^{1/2} N^{\frac{1+\epsilon}{3}} +
\int_K^\infty ||u^1||_{L^3} ||\partial^4 e^{t\Delta} \partial^{-8/3}
u^2||_{L^3} dt \\ \nonumber
& \lesssim ||u^1||_{L^3} ||u^2||_{L^3} K^{1/2} N^{\frac{1+\epsilon}{3}} +
\int_K^\infty ||u^1||_{L^3} ||u^2||_{L^3} t^{-2} N^{\frac{-8(1+\epsilon)}{3}}
dt \\ \nonumber
& \lesssim ||u^1||_{L^3} ||u^2||_{L^3} (K^{1/2} N^{\frac{1+\epsilon}{3}} +
K^{-1} N^{\frac{-8(1+\epsilon)}{3}}) 
\end{align}

Here, we use Holder inequality in the second line, Bernstein's inequality and
fact that the heat flow does not alter the frequency of the data in the
first term of the third line, and Sobolev inequality in the second term of the
third line.  Berstein's inequality is used once again in the second term of the
fifth line.  

The quantity inside the bracket is minimized at the natural scaling
$K = N^{-2(1+\epsilon)}$, substituting in, we have the estimate:

\begin{displaymath}
 \int_0^\infty ||e^{t\Delta}u^1 \partial_1 e^{t\Delta}u^2||_{L^3} dt \lesssim
||u^1||_{L^3} ||u^2||_{L^3} N^{-2/3 - 2\epsilon/3}.
\end{displaymath}

\textbf{Case 2} $|\text{supp}(\widehat{u^1})| \subset
[N^{1+\epsilon},O(1)N^{1+\epsilon}]$, $|\text{supp}(\widehat{u^2})| \subset
[N,O(1)N]$
\begin{align} \nonumber
\qquad &\int_0^\infty ||e^{t\Delta}u^1 \partial_1 e^{t\Delta}u^2||_{L^3} dt\\
\nonumber 
&\leq \int_0^\infty ||e^{t\Delta}u^1||_{L^3_{x_2}L^\infty_{x_3}}
||\partial_1 e^{t\Delta}u^2||_{L^3} dt \\ \nonumber
&\lesssim \int_0^K N^{(1+\epsilon)/3} ||u^1||_{L^3} t^{-1/2} ||u^2||_{L^3} dt +
\int_K^\infty ||\partial^{1/3} e^{t\Delta}u^1||_{L^3} t^{-1/2} ||u^2||_{L^3} dt
\\ \nonumber
& \lesssim  ||u^1||_{L^3}||u^2||_{L^3} N^{(1+\epsilon)/3} K^{1/2} +
\int_K^\infty ||\partial^{3} e^{t\Delta} \partial^{-8/3} u^1||_{L^3} t^{-1/2}
||u^2||_{L^3} dt \\ \nonumber
& \lesssim ||u^1||_{L^3}||u^2||_{L^3} (N^{(1+\epsilon)/3} K^{1/2} +
\int_K^\infty t^{-2} N^{-8(1+\epsilon)/3} dt) \\ \nonumber
& \lesssim ||u^1||_{L^3}||u^2||_{L^3} (N^{(1+\epsilon)/3} K^{1/2} + K^{-1}
N^{-8(1+\epsilon)/3})
\end{align}

From here we proceed as in case 1 to obtain the same conclusion.  As mentioned
above, the other estimates where one of the term is of low frequency can be
obtained using similar methods.

To summarize, we have proven the following:

\begin{prop}
The following estimates hold for the inhomogenous term $F$ for equation
(\ref{R})
\begin{equation} \label{eqn:F1}
 ||F||_{L^1_tL^3_x} \leq C_1||u^i||_{L^3} ||u^j||_{L^3} N^{-2/3 - 2\epsilon/3}
\end{equation}
if one of $u^i$ or $u^j$ has frequency support at scale $N^{1+\epsilon}$. 
Otherwise, we have
\begin{equation} \label{eqn:F2}
 ||F||_{L^1_tL^3_x} \leq C_1 ||u^i||_{L^3} ||u^j||_{L^3} N^{-2/3}
\end{equation}
\end{prop}

\section{Global wellposedness for the residual system}
In this section, we would like to prove global wellposedness for the residual
system (\ref{R}), thus completing the proof of the main theorem:

\begin{prop}
 Assuming the conditions for the initial data listed in theorem
(\ref{thm:main}), the system (\ref{R}) has a unique global solution in
$C_b(\mathbb{R}_+; L^3)$
\end{prop}

The proof of it makes use of the following theorem and proposition proven in
\cite{CG1}:

\begin{thm} \label{thm:CG}
 Let $p \in (3,\infty)$ be given.  There is a constant $C_0 > 0$ such that for
any $R_0$ in $B^{-1+\frac{3}{p}}_{p,2}$, $F$ in $L^1(\mathbb{R}_+;
B^{-1+\frac{3}{p}}_{p,2})$ and $v^\star$ in $L^2(\mathbb{R}_+; L^\infty)$
satisfying
\begin{equation} \label{eqn:cond}
 ||R_0||_{B^{-1+\frac{3}{p}}_{p,2}} + ||F||_{L^1(\mathbb{R}_+;
B^{-1+\frac{3}{p}}_{p,2})} \leq C_0^{-1}e^{-C_0||v^\star||^2_{L^2(\mathbb{R}_+;
L^\infty)}}
\end{equation}
there is a unique global solution $R$ to \emph{(\ref{R})} associated with $R_0$
and $F$, such that 
\begin{displaymath}
 R\in C_b(\mathbb{R}_+; H^{1/2}) \cap L^2(\mathbb{R}_+; H^{3/2})
\end{displaymath}
\end{thm}

\begin{proof}
It is easy to check that the assumptions on the initial data are enough to
satisfy (\ref{eqn:cond}).  Indeed we have, firstly,
\begin{displaymath}
 ||v^\star||^2_{L^2(\mathbb{R}_+; L^\infty)} \leq \sum_{1\leq i \leq
3}||u_0^i||^2_{B^{-1}_{\infty,2}}
\end{displaymath}
by definition.  From the embedding of $L^2$ into $B^{-1}_{\infty,2}$, we have 
\begin{displaymath}
 ||u_0^i||_{B^{-1}_{\infty,2}} \lesssim ||u_0^i||_{L^2}
\end{displaymath}
which in turn is bounded by $(C_0^{-1} \delta \log
N^{1+\epsilon})^{1/2}$ by our assumptions.  Therefore we have 
\begin{displaymath}
 e^{-C_0||u^\star||^2_{L^2(\mathbb{R}_+; L^\infty)}} \geq
N^{-\delta(1+\epsilon)}
\end{displaymath}

On the other hand, since $L^3(\mathbb{T}^3) \hookrightarrow
B^{-1+\frac{3}{p}}_{p,2}(\mathbb{T}^3)$ for $p > 3$, we have 
\begin{displaymath}
 ||F||_{L^1(\mathbb{R}_+; B^{-1+\frac{3}{p}}_{p,2})} \leq
C_2 ||F||_{L^1(\mathbb{R}_+; L^3)}
\end{displaymath}

From the assumption of the energy of the initial condition, and by Bernstein's
inequality, $||u_0^1||_{L^3} \lesssim N^{\frac{1}{3}}(C^{-1} \delta \log
N^{1+\epsilon})^{1/2}$ and $||u_0^2|| \lesssim N^{\frac{1+\epsilon}{3}}(C^{-1}
\delta \log N^{1+\epsilon})^{1/2}$.  Together with the assumption on the $L^3$
norm of $||u_0^3||$, we have the following:

For $i\neq2$,
\begin{displaymath}
 ||u_0^i||_{L^3}||u_0^2||_{L^3} \lesssim N^{\frac{2+\epsilon}{3}} C^{-1} \delta
\log N^{1+\epsilon}
\end{displaymath}
and
\begin{displaymath}
 ||u_0^1||_{L^3}||u_0^3||_{L^3} \lesssim (C^{-1} \delta)^{1/2} N^{\frac{2}{3} -
\delta(1+\epsilon)}
\end{displaymath}

From the estimates (\ref{eqn:F1}) and (\ref{eqn:F2}), and for sufficiently
small $\delta$, say less than $\epsilon/4$, we can conclude that
\begin{displaymath}
||F||_{L^1(\mathbb{R}_+; L^3)} \leq C^{-1} N^{-\delta(1+\epsilon)}
\end{displaymath}
\end{proof}

Condition (\ref{eqn:cond}) is thus satisfied and the wellposedness of the
residual equation follows.  Global wellposedness of the original problem
(\ref{NS}) is obtained by combining the wellposedness of (\ref{R}) and that of
heat equations.

\section{An illustration of an example}
The largest critical space known for the 3D Navier-Stokes equations is the Besov
space $B^{-1}_{\infty, \infty}$.  Bourgain and Pavlovi\'c \cite{BP} have shown
that the Navier-Stokes system is illposed in the critical space.  We shall show
that the conditions set out for initial data in the main theorem allows
arbitrary large $B^{-1}_{\infty, \infty}$ norm.  Chemin, Gallagher, Paicu and
Zhang (e.g. \cite{CG1}) have also provided different examples of arbitrarily
large initial data in $B^{-1}_{\infty, \infty}$ that the 3D Navier-Stokes system
is globally wellposed under those data.

Let us first recall the theorem we are trying to establish:
\begin{thm}
 For any $M > 0$, there exists initial data $u_0^1, u_0^2, u_0^3$ satisfying
that conditions in Theorem \ref{thm:main} such that each component $u_0^i$ is
large in the following sense:
\begin{displaymath}
 ||u_0^i||_{B^{-1}_{\infty,\infty}} > M
\end{displaymath}
\end{thm}

\begin{rmk}
 Notice that by the nature of our initial condition being a two dimensional
object, and from the embedding that in 2D, $L^2 \hookrightarrow B^{-1}_{\infty,
\infty}$, the data demonstarted below will also have large energy.
\end{rmk}

\begin{proof}
 Consider the function $u_0^1$ with frequency support on ${0} \times
[N,2N]\times[N,2N]$, we can write it in terms of its Fourier expansion
\begin{displaymath}
 u_0^1(x_2,x_3) = \sum_{m_2=N}^{2N} \sum_{m_3=N}^{2N} a_{m_2m_3}
e^{im_2x_2}e^{im_3x_3}
\end{displaymath}
We shall pick the values $a_{m_2m_3}$ such that at time $t = 1/{N^2}$, they are
all equal, i.e.
\begin{displaymath}
 a_{m_2m_3} e^{-\frac{m_2^2+m_3^2}{N^2}} = A
\end{displaymath}
for some $A$.
We shall also require that the $L^2$ norm of $u_0^1$ to be equal to $C^{-1}
\log N$, i.e.
\begin{displaymath}
 \sum_{m_2=N}^{2N} \sum_{m_3=N}^{2N} a_{m_2m_3}^2 = C^{-1} \log N
\end{displaymath}
It is clear that we can find coefficients $a_{m_2m_3}$ that satisfy the above
conditions.

For $u_0^2$, we pick a similar function, but localized in frequency on the
support $[N^{1+\epsilon}, 2N^{1+\epsilon}] \times {0} \times [N^{1+\epsilon},
2N^{1+\epsilon}]$ instead.  More precisely, if we write
\begin{displaymath}
 u_0^2(x_1,x_3) = \sum_{m_1={N^{1+\epsilon}}}^{2{N^{1+\epsilon}}}
\sum_{m_3={N^{1+\epsilon}}}^{2{N^{1+\epsilon}}} b_{m_1m_3}
e^{im_2x_2}e^{im_3x_3},
\end{displaymath}
then the coefficients $b_{m_1m_3}$ satisfies the following:
\begin{displaymath}
 b_{m_1m_3} e^{-\frac{m_1^2+m_3^2}{N^{2(1+\epsilon)}}} = B
\end{displaymath}
for some $B$ and
\begin{displaymath}
 \sum_{m_1=N^{1+\epsilon}}^{2N^{1+\epsilon}}
\sum_{m_3=N^{1+\epsilon}}^{2N^{1+\epsilon}} b_{m_1m_3}^2 = C^{-1}
\log N^{1+\epsilon}
\end{displaymath}

Lastly, we shall take $u_0^3 = (C^{-1} \delta \log N^{1+\epsilon})^{1/2}e^{ix_1}e^{ix_2}$.

By construction, the data $(u_0^1,u_0^2,u_0^3)$ satisfies the conditions of
theorem (\ref{thm:main}), the only task remaining is to compute the
$B^{-1}_{\infty,\infty}$ norm of the data.

\begin{align} \nonumber
 ||e^{\frac{\Delta}{N^2}} u_0^1||_{L^\infty}  &\geq \sum_{m_2=N}^{2N}
\sum_{m_3=N}^{2N} a_{m_2m_3} e^{-\frac{m_2^2+m_3^2}{N^2}} \\ \nonumber
&= N(\sum_{m_2=N}^{2N} \sum_{m_3=N}^{2N} a^2_{m_2m_3}
e^{-\frac{2(m_2^2+m_3^2)}{N^2}})^{1/2} \\ \nonumber
&\gtrsim N (\sum_{m_2=N}^{2N} \sum_{m_3=N}^{2N} a^2_{m_2m_3})^{1/2} \\
\nonumber
&= N ||u_0^1||_{L^2} \\ \nonumber
&= C^{-1/2} N (\log N)^{1/2}
\end{align}

On the first equality above, we use the condition that at $t = N^{-2}$, the summands are all equal.

By definition,
\begin{align} \nonumber
 ||u_0^1||_{B^{-1}_{\infty,\infty}} &= \sup_{t>0} t^{1/2} ||e^{t\Delta}
u_0^i||_{L^\infty} \\ \nonumber
&\geq N^{-1} ||e^{\frac{\Delta}{N^2}} u_0^1||_{L^\infty} \\ \nonumber
&\geq C^{-1/2}(\log N)^{1/2}
\end{align}

Similarly, we can show that $||u_0^2||_{B^{-1}_{\infty,\infty}} \geq C^{-1/2}
(\log N^{1+\epsilon})^{1/2}$

Lastly, we can compute $||u_0^3||_{B^{-1}_{\infty,\infty}} = (C^{-1} \delta
\log N^{1+\epsilon})^{1/2}$.  By taking $N$ to be large, we can guarantee that
the $B^{-1}_{\infty, \infty}$ norm of all three components to be larger than
$M$ for any $M$

\end{proof}

\section{Concluding Remarks}
By considering initial data with conditions in the frequency support, we can reduce and control the interaction between different components coming from the nonlinear term of the equation, and allows a perturbative approach as the equations are now closed to a decoupled set of heat equations.  It will be interesting to consider a relaxation of such conditions, for example, when the initial data is the sum of two 2D data, say one in the $x-y$ direction and the other in the $y-z$ direction, with some amount of frequency separation.  The reduction in this case would be to the established results of global wellposedness of 2D Navier-Stokes instead of the heat equations.  Further considerations such as above will be addressed in a future article.

\section{Acknowledgements}
The author would like to thank Prof. Sergiu Klainerman for suggesting the
problem and Prof. Sergiu Klainerman and Prof. Igor Rodnianski for providing
useful disussions and insights, without which the paper would not be possible. 
The author would also like to thank Jonathan W. Luk for proofreading early
versions of this paper and his useful suggestions of various modifications and
additions.

\end{document}